\documentclass[10pt, reqno]{amsart}

\usepackage{hyperref, cite, xcolor}
\usepackage{amsmath}
\usepackage{amsthm}
\usepackage{amssymb}
\usepackage{amsfonts}
\usepackage{latexsym}
\usepackage{setspace}
\usepackage{enumitem}
\setlist[itemize]{leftmargin=*}
\setlist[enumerate]{leftmargin=*}

\usepackage[font=footnotesize,labelfont=bf]{caption}

\newtheorem{theorem}{Theorem}

\newtheorem{lemma}[equation]{Lemma}

\theoremstyle{definition}

\numberwithin{theorem}{section}
\numberwithin{equation}{section}

\renewcommand{\phi}{\varphi}

\newcommand{\E}{\mathcal{E}}

\renewcommand{\pmod}[1]{\,(\operatorname{mod} #1)}
\newcommand{\lcm}{\operatorname{lcm}}

\renewcommand{\geq}{\geqslant}
\renewcommand{\leq}{\leqslant}

\let\oldenumerate=\enumerate
	\def\enumerate{
	\oldenumerate
	\setlength{\itemsep}{5pt}
	}
\let\olditemize=\itemize
	\def\itemize{
	\olditemize
	\setlength{\itemsep}{5pt}
	}

\allowdisplaybreaks

\begin{document}

\title[the Euler totient function near prime arguments]{On the difference in values of the Euler totient function near prime arguments}

\author[S.R.~Garcia]{Stephan Ramon Garcia}
\address{Department of Mathematics\\Pomona College\\610 N. College Ave., Claremont, CA 91711} 
\email{stephan.garcia@pomona.edu}
\urladdr{http://pages.pomona.edu/~sg064747}
\thanks{SRG supported by NSF grant DMS-1265973,
a David L. Hirsch III and Susan H. Hirsch Research Initiation Grant, 
and the Budapest Semesters in Mathematics (BSM)
Director's Mathematician in Residence (DMiR) program.}

\author[F.~Luca]{Florian Luca}
\address{School of Mathematics\\University of the Witwatersrand\\Private Bag 3, Wits 2050, Johannesburg, South Africa\\
Max Planck Institute for Mathematics, Vivatsgasse 7, 53111 Bonn, Germany\\
Department of Mathematics, Faculty of Sciences, University of Ostrava, 30 dubna 22, 701 03
Ostrava 1, Czech Republic}
\email{Florian.Luca@wits.ac.za}
\thanks{F. L. was supported in part by grants CPRR160325161141 and an A-rated researcher award both from the NRF of South Africa and by grant no. 17-02804S of the Czech Granting Agency. }

\begin{abstract}
We prove unconditionally that for each $\ell \geq 1$, the difference $\phi(p-\ell) - \phi(p+\ell)$ is positive for $50\%$ of odd primes $p$ and negative for $50\%$.
\end{abstract}



\maketitle

\section{Introduction}

In what follows, $p$ always denotes an odd prime number. 
The inequality $\phi(p-1) \geq \phi(p+1)$ appears to hold
for an overwhelming majority of twin primes $p,p+2$, 
and to be reversed for small, but positive, proportion of the twin primes \cite{GKL}. 
To be more specific,
if the Bateman--Horn conjecture is true, then
the inequality above holds for at least 65.13\% of twin prime pairs and is reversed for at least $0.47\%$ of pairs.
Numerical evidence suggests, in fact, that the ratio is something like 98\% to 2\%.  
In other words, for an overwhelming majority of twin prime pairs $p,p+2$, it appears that the first prime
has more primitive roots than does the second.

Based upon numerical evidence, it was conjectured in \cite{GKL} that this bias disappears
if only $p$ is assumed to be prime.  That is,
$\phi(p-1) > \phi(p+1)$ for $50\%$ of primes and the inequality is reversed for $50\%$ of primes.
We prove this unconditionally and, moreover, we are able to handle wider spacings as well.
If all primality assumptions are dropped, then it is known that $\phi(n-1) > \phi(n+1)$ asymptotically
$50\%$ of the time.  This follows from work of Shapiro, who considered the distribution
function of $\phi(n) / \phi(n-1)$ \cite{Shapiro}. 

Let $\pi(x)$ denote the number of primes at most $x$ and let $\sim$ denote asymptotic equivalence.
The Prime Number Theorem ensures that $\pi(x) \sim x/\log x$.
Our main theorem is the following.

\begin{theorem}\label{Theorem:Main}
Let $\ell$ be a positive integer.  As $x \to \infty$ we have the following:
\begin{enumerate}\addtolength{\itemsep}{5pt}
\item $\# \{ p \leq x \,:\, \phi(p-\ell)>\phi(p+\ell) \} \,\sim\, \frac{1}{2}\pi(x)$.
\item $\# \{ p \leq x \,:\, \phi(p-\ell)<\phi(p+\ell) \} \,\sim\, \frac{1}{2}\pi(x)$.
\item $\# \{ p \leq x \,:\, \phi(p-\ell)= \phi(p+\ell) \} \,=\, o(\pi(x))$.
\end{enumerate}
\end{theorem}   

A curious phenomenon occurs in (c), in the sense that the decay rate relative to $\pi(x)$ depends
upon $\ell$ in a peculiar manner.  Theorem \ref{Theorem:Equality} shows that
\begin{equation*}
\# \big\{ p \leq x : \phi(p-\ell) = \phi(p+\ell) \big\} 
\,\,\ll \,\,
\begin{cases}
\dfrac{x}{(\log x)^3} & \text{if $\ell = 4^n-1$},\\[10pt]
\dfrac{x}{e^{(\log x)^{1/3}}} & \text{otherwise}.
\end{cases}
\end{equation*}
This does not appear to be an artifact of the proof since it is borne out in numerical computations
(see Table \ref{Table:Equality})
and is consistent with the Bateman--Horn conjecture.

\begin{table}\small
\begin{equation*}
\begin{array}{|c@{\hskip 8pt}c|c@{\hskip 8pt}c|c@{\hskip 8pt}c|c@{\hskip 8pt}c|c@{\hskip 8pt}c|c@{\hskip 8pt}c|c@{\hskip 8pt}c|c@{\hskip 8pt}c|}
\hline
\ell & \# & \ell & \# & \ell & \# & \ell & \# & \ell & \# & \ell & \# & \ell & \# & \ell & \# \\
\hline
 1 & 103 & 9 & 359 & 17 & 106 & 25 & 6 & 33 & 338 & 41 & 109 & 49 & 38 & 57 & 295 \\
 2 & 49 & 10 & 4 & 18 & 219 & 26 & 47 & 34 & 47 & 42 & 322 & 50 & 5 & 58 & 54 \\
 \bf 3 & \bf 201078 & 11 & 107 & 19 & 104 & 27 & 357 & 35 & 3 & 43 & 121 & 51 & 371 & 59 & 127 \\
 4 & 58 & 12 & 214 & 20 & 3 & 28 & 17 & 36 & 374 & 44 & 39 & 52 & 38 & 60 & 538 \\
 5 & 5 & 13 & 98 & 21 & 403 & 29 & 117 & 37 & 97 & 45 & 486 & 53 & 126 & 61 & 126 \\
 6 & 231 & 14 & 7 & 22 & 52 & 30 & 507 & 38 & 45 & 46 & 47 & 54 & 303 & 62 & 45 \\
 7 & 43 & \bf 15 & \bf 108772 & 23 & 136 & 31 & 98 & 39 & 380 & 47 & 124 & 55 & 2 & \bf 63 & \bf 22654 \\
 8 & 50 & 16 & 39 & 24 & 301 & 32 & 53 & 40 & 5 & 48 & 236 & 56 & 6 & 64 & 48 \\
\hline
\end{array}
\end{equation*}
\caption{The number $(\#)$ of primes $p \leq 2{,}038{,}074{,}743$ (the hundred millionth prime) for which
$\phi(p-\ell) = \phi(p+\ell)$.  This number is exceptionally
large if $\ell=4^n - 1$; see Theorem \ref{Theorem:Equality} for an explanation.}
\label{Table:Equality}
\end{table}

We first prove Theorem \ref{Theorem:Main} in the case $\ell = 1$.
This is undertaken in Section \ref{Section:Proof} and it comprises the bulk of this article.
For the sake of readability, we break the proof into several steps which we hope are easy to follow.
In Section \ref{Section:General}, we outline
the modifications necessary to treat the case $\ell \geq 2$.
This approach permits us to focus on the main ingredients that are common to
both cases, without getting sidetracked by all of the adjustments necessary
to handle the general case.

\section{Proof of Theorem \ref{Theorem:Main} for $\ell = 1$}\label{Section:Proof}

\subsection{The case of equality}
Our first job is to show that the set of primes $p$ for which $\phi(p-1) = \phi(p+1)$
has a counting function that is $o(\pi(x))$.
We need the following lemma, which generalizes
earlier work by Erd\H{o}s, Pomerance, and S\'ark\"ozy \cite{Erdos} in the case $k=1$.
The upper bound (b) in the following was strengthened in a preprint of Yamada \cite{Yamada}.

\begin{lemma}[Graham--Holt--Pomerance \cite{GHP}]\label{Lemma:GHP}
Suppose that $j$ and $j+k$ have the same prime factors.
Let $g = \gcd(j,j+k)$ and suppose that
\begin{equation}\label{eq:qr}
 \frac{jt}{g} + 1 \qquad \text{and} \qquad \frac{(j+k)t}{g} + 1
\end{equation}
are primes that do not divide $j$.
\begin{enumerate}
\item Then $\displaystyle n = j \left( \frac{(j+k)t}{g} + 1 \right)$ satisfies $\phi(n) = \phi(n+k)$.
\item For $k$ fixed and sufficiently large $x$,
the number of solutions $n \leq x$ to $\phi(n) = \phi(n+k)$ that are not of the form above
is less than $x/\exp( (\log x)^{1/3})$.
\end{enumerate}
\end{lemma}

Consider the case $k=2$ and $n = p-1$, in which $p$ is prime.  
Suppose that $j$ and $j+2$ have the same prime factors and let $g = \gcd(j,j+2)$.
Let us also suppose that $t$ is a positive integer such that
\begin{equation*}
\frac{jt}{g}+1\qquad \text{and}\qquad \frac{(j+2)t}{g}+1
\end{equation*}
are primes and 
\begin{equation*}
p-1=j\left(\frac{(j+2)t}{g}+1\right).
\end{equation*}
Since $j$ and $j+2$ have the same prime factors, they are both powers of $2$.
Then $j=2$ and $j+2 = 4$, so $g = 2$.  Consequently,
\begin{equation}\label{eq:PT}
t+1, \qquad 2t+1,\quad \text{and} \quad 4t+3
\end{equation}
are prime.  Reduction modulo $3$ reveals that at least one of them
is a multiple of $3$.  The only prime triples produced by \eqref{eq:PT} are $(2,3,7)$ and $(3,5,11)$,
in which $r=1$ and $r=2$, respectively.  Consequently,
\begin{equation}\label{eq:PP11Pix}
\#\big\{ p \leq x : \phi(p-1) = \phi(p+1) \big\} < \frac{x}{\exp( (\log x)^{1/3})}  + 2 = o(\pi(x)).
\end{equation}
This is Theorem \ref{Theorem:Main}.c in the case $\ell = 1$.

\subsection{A comparison lemma}

Instead of comparing $\phi(p-1)$ and $\phi(p+1)$ directly, it is more convenient to compare the related quantities
\begin{equation}\label{eq:DiffExp}
\frac{\phi(p-1)}{p-1} = \prod_{q|(p-1)} \left( 1 - \frac{1}{q} \right)
\qquad\text{and}\qquad
\frac{\phi(p+1)}{p+1} = \prod_{q|(p+1)} \left( 1 - \frac{1}{q} \right),
\end{equation}
in which $q$ is prime.  Let 
\begin{equation}\label{eq:Sp}
S(p) := \frac{\phi(p-1)}{p-1} - \frac{\phi(p+1)}{p+1},
\end{equation}
which we claim is nonzero for $p \geq 5$.  
Let $P(n)$ denote the largest prime factor of $n$. Since 
\begin{equation}\label{eq:qqpn}
\frac{\phi(n)}{n} = \prod_{q|n} \bigg(\frac{q-1}{q}\bigg),
\end{equation}
it follows that $P(n)$ is the largest prime factor of the denominator of $\phi(n)/n$.
Since $\gcd(p-1,p+1) = 2$, the condition $S(p) = 0$ implies that $p-1$ and $p+1$ are both
powers of $2$.  Thus, $S(p) = 0$ holds only for $p=3$.

Something similar to the following lemma is in \cite{GKL}, although there
it was assumed that $p+2$ is also prime.  The adjustment
for $\ell \geq 2$ is discussed in Section \ref{Section:General}.

\begin{lemma}[Comparison Lemma]\label{Lemma:Sp}
The set of primes $p$ for which
$\phi(p-1) - \phi(p+1)$ and $S(p)$ 
have the same sign has counting function asymptotic to $\pi(x)$ as $x \to \infty$.
\end{lemma}

\begin{proof}
In light of \eqref{eq:PP11Pix},
it suffices to show that
\begin{equation}\label{eq:Trouble1}
\phi(p-1) \,> \,\phi(p+1)
\quad\iff\quad
\frac{\phi(p-1)}{p-1}\, > \,\frac{\phi(p+1)}{p+1}
\end{equation}
on a set of full density in the primes.
The forward direction is clear, so we focus on the reverse.
If the inequality on the right-hand side of \eqref{eq:Trouble1} holds, then
\begin{align}
0 
&< p \big(\phi(p-1) - \phi(p+1)\big) + \phi(p-1) + \phi(p+1) \nonumber\\
&\leq p \big(\phi(p-1) - \phi(p+1)\big) + \tfrac{1}{2}(p-1) + \tfrac{1}{2}(p+1) \nonumber\\
&= p \big(\phi(p-1) - \phi(p+1) + 1\big)\label{eq:GoingWrong}
\end{align}
because $p-1$ and $p+1$ are even.
Since $\phi(n)$ is even for $n \geq 3$, it follows that $\phi(p-1) - \phi(p+1) \geq 0$.
Now appeal to \eqref{eq:PP11Pix} to see that strict inequality holds on a set of
full density in the primes.
\end{proof}

\subsection{Some preliminaries}
In our later study of the quantity $S(p)$, we need to 
avoid four classes of inconvenient primes.
To make the required estimates, we need some notation.
Let $x$ be large, let $y:=\log\log x$, and define
\begin{equation}\label{eq:Ly}
L_y := \lcm \{ m: m\leq y\}.
\end{equation}
Then $L_y= e^{\psi(y)}$, in which
\begin{equation*}
\psi (y):=\sum_{p^k\leq y}\log p 
\end{equation*}
is Chebyshev's function. 
Since the Prime Number Theorem
asserts that $\psi(y) = y + o(y)$ as $y \to \infty$, we obtain
\begin{equation}\label{eq:Lxy2}
L_y = e^{y+o(y)} < e^{2y}=(\log x)^2
\end{equation}
for sufficiently large $x$. 
For a positive integer $n$, 
let $D_y(n)$ denote the largest divisor of $n$ that is \emph{$y$-smooth}:
\begin{equation}\label{eq:Dyn}
D_y(n):=\max\big\{d: \text{$d| n$ and $P(d)\leq y$}\big\}.
\end{equation}

On occasion, we will need the Brun sieve.
Let $f_1,f_2,\ldots,f_m$ be a collection of distinct irreducible polynomials 
with positive leading coefficients.  An integer $n$ is \emph{prime generating} for 
this collection if each $f_1(n), f_2(n),\ldots, f_m(n)$ is prime. 
Let $G(x)$ denote
the number of prime-generating integers at most $x$ and
suppose that $f = f_1f_2\cdots f_m$ does not vanish identically modulo any prime.
As $x \to \infty$,
\begin{equation*}
G(x)
\ll  \frac{x}{(\log x)^m}, 
\end{equation*}
in which the implied constant depends only upon $m$ and $\prod_{i=1}^m \deg f_i$
\cite[Thm.~3, Sect.~I.4.2]{Tenenbaum}.
The upper bound obtained in this manner has the same order of magnitude
as the prediction furnished by the Bateman--Horn conjecture \cite{Bateman}.

\subsection{Inconvenient primes of Type 1}
Let 
\begin{equation*}
\E_1(x) := \big\{p\leq x: D_y(p^2-1)\nmid L_y \big\}.
\end{equation*}
We will prove that
\begin{equation}\label{eq:E1}
     \#\E_1(x)\ll \frac{x}{y^{1/2} \log x}=o(\pi(x)).
\end{equation}

Suppose that $p\in \E_1(x)$.
Then \eqref{eq:Dyn} and the definition of $\E_1(x)$ yield 
a divisor $d$ of $p^2-1$ such that $P(d)\leq y$ and $d\nmid L_y$. 
These conditions provide a prime power $q^b$ with least exponent $b$ such that
\begin{equation}\label{eq:qPby}
q^b | d,  \qquad q=P(q^b)\leq y,  \quad \text{and} \quad y < q^b .
\end{equation}
Indeed, if every $y$-smooth prime power $q^b$ that divides $d$
satisfies $q^b \leq y$, then \eqref{eq:Ly} would imply that $d|L_y$,
a contradiction.  We also observe that
the second two conditions in \eqref{eq:qPby} ensure that $b\geq 2$. 

We claim that for large $x$ either $p-1$ or $p+1$ has a prime power divisor $q^c$ with $c\geq 2$ in the interval $[y/2, y^2]$.
Since
\begin{equation*}
p^2-1=(p-1)(p+1) \quad \text{and} \quad  \gcd(p-1,p+1)=2,
\end{equation*}
it follows that $d/(2,d)$ divides one of
$p+1$, $p-1$.
There are two cases to consider.
\begin{itemize}

\item If $q = 2$, then $2^{b-1}$ divides one of $p-1$, $p+1$.  
Since we aim to prove \eqref{eq:E1} as $x \to \infty$,
we may assume that $b \geq 3$ since for $b=2$, the third statement in \eqref{eq:qPby}
implies $\log \log x < 4$.
Next observe that \eqref{eq:qPby} implies that $y/2 < 2^{b-1}$.
The minimality of $b$ in \eqref{eq:qPby} 
ensures that $2^{b-1} \leq y< y^2$.  Thus,
$2^{b-1} \in [\frac{y}{2},y^2]$ has $b-1\geq 2$ (since $b \geq 3$), and divides one of $p-1$, $p+1$.

\item If $q$ is odd, then $q^{b}$ divides $p-1$ or $p+1$.
The minimality of $b$ ensures that
$q^{b-1} \leq y$ and the second statement in \eqref{eq:qPby} yields
\begin{equation*}
\frac{y}{2} < y < q^{b} = q^{b-1} q \leq y^2,
\quad \text{and hence} \quad 
q^{b} \in [ \tfrac{y}{2}, y^2]
\quad\text{with $b \geq 2$}.
\end{equation*}
\end{itemize}
For large $x$, we conclude that one of $p-1$, $p+1$ has a prime power divisor $q^c$ with $c\geq 2$ in $[y/2,y^2]$. 

Let $\pi_{s}(x)$ denote the number of prime powers $p^a$ with $a\geq 2$ that are at most $x$.  Since $p^a\leq x$ with $a\geq 2$ implies either
$a=2$ and $p\leq x^{1/2}$, or $a\in [3,(\log x/\log 2)]$ and $p\leq x^{1/3}$, the Prime Number Theorem implies that
\begin{equation*}
\pi_{s}(x) =\pi({\sqrt{x}})+O\big(\pi(x^{1/3})\log x\big)=(2+o(1))\frac{x^{1/2}}{\log x}
\end{equation*}
as $x\to\infty$.  
Let $\pi(x;m,k)$ denote the number of primes at most $x$ that are congruent to $k$ modulo $m$.
Then the Brun sieve implies
\begin{align*}
\# \E_1(x) 
& \leq  \sum_{\substack{q^b\in [y/2,y^2]\\ b \geq 2}} \pi(x;q^b, 1) + \sum_{\substack{q^b\in [y/2,y^2]\\ b \geq 2}} \pi(x;q^b, -1)\\
&\ll  \sum_{\substack{q^b\in [y/2,y^2]\\ b \geq 2}} \frac{x}{\phi(q^b) \log x}
 \ll  \sum_{\substack{q^b\in [y/2,y^2]\\ b\geq 2}} \frac{2x}{q^b \log x}\\
& \ll  \frac{x}{\log x} \sum_{\substack{q^b\in [y/2,y^2]\\ b\geq 2}} \frac{1}{q^b}\\
& \ll  \frac{x}{\log x}  \int_{y/2}^{y^2} \frac{d\pi_s(t)}{t} \leq   \frac{x}{\log x}  \int_{y/2}^{\infty} \frac{d\pi_s(t)}{t}  \\
& \ll  \frac{x}{\log x} \left( \frac{\pi_s(t)}{t} \Bigg|_{y/2}^{\infty} + \int_{y/2}^{\infty} \frac{\pi_s(t)\,dt}{t^2}\right)\\
& \ll  \frac{x}{\log x} \left( \frac{(y/2)^{1/2}}{(y/2) \log(y/2)}+ \int_{y/2}^{\infty} t^{-3/2} (\log t)^{-1/2}\,dt\right)\\
& \ll   \frac{x}{y^{1/2} (\log y) (\log x)}
=o(\pi(x))
\end{align*} 
as $x\to\infty$.  This is the desired estimate \eqref{eq:E1}.

\subsection{Inconvenient primes of Type 2}
Fix a large $x$ and
define the function
\begin{equation*}
    h_y(n):=\sum_{\substack{r|  n\\ r> y}} \frac{1}{r},
\end{equation*}
in which $r$ is prime.  Let
\begin{equation*}
\E_2(x)
:=
\left\{p\leq x\,:\, h_y(p-1)>\frac{1}{y{\sqrt{\log y}}} \quad \text{or} \quad  h_y(p+1)>\frac{1}{y{\sqrt{\log y}}}\right\}.
\end{equation*}
We claim that
\begin{equation}\label{eq:E2}
     \#\E_2(x)\ll \frac{\pi(x)}{\sqrt{\log y}}=o(\pi(x))
\end{equation}
as $x\to\infty$. 
This will follow from an averaging argument similar to \cite[Lem.~3]{LuPo}. 

The Brun sieve with $f(t) = rt\pm 1$ provides
\begin{equation*}
 \pi(x;r,\pm 1)   \ll  \frac{\pi(x)}{r} \qquad \text{for $y\leq r\leq (\log x)^3$}
\end{equation*}
uniformly for $r$ in the specified range \cite[Thm.~3, Sect.~I.4.2]{Tenenbaum}.
We use the  trivial estimate
\begin{equation*}
 \pi(x;r,\pm 1)  \leq  \frac{x}{r}  \qquad \text{for $(\log x)^3\leq r\leq x$}.
\end{equation*}
We also require the upper bound
\begin{equation}\label{eq:yly}
\sum_{r \geq y} \frac{1}{r^2} 
= \int_y^{\infty}\frac{d\pi(t)}{t^2} = \frac{\pi(t)}{t^2} \Bigg|_y^{\infty} + 2 \int_y^{\infty} \frac{\pi(t)\,dt}{t^3}
\ll \frac{\pi(y)}{y^2} \ll \frac{1}{y \log y} ,
\end{equation}
which is afforded by the Prime Number Theorem.  
As $x \to \infty$, we have
\begin{align*}
\sum_{p\leq x} h_y(p \pm 1) 
&=  \sum_{p\leq x}\, \sum_{\substack{r|  (p\pm 1)\\ r >  y}} \frac{1}{r} 
=  \sum_{y\leq r\leq x} \frac{1}{r}\sum_{\substack{p\leq x\\ r|  (p\pm 1)}} 1 \\
&=  \sum_{y\leq r\leq x} \frac{\pi(x;r,\mp 1)}{r} \\
&  \ll   \sum_{y\leq  r\leq (\log x)^3} \frac{\pi(x)}{r^2} \quad +\sum_{(\log x)^3<r\leq x} \frac{x}{r^2}\\
&  \ll   \pi(x)\sum_{r\geq y} \frac{1}{r^2}\quad +\quad x\sum_{r\geq (\log x)^3} \frac{1}{r^2} \\
& \ll   \frac{\pi(x)}{y\log y} +\frac{x}{(\log x)^3 }
 \ll   \frac{\pi(x)}{y\log y} +\frac{\pi(x)}{(\log x)^2 }\\
&\ll \frac{\pi(x)}{y \log y }.
\end{align*}
Consequently,
\begin{align*}
\frac{\#\E_2(x)}{y \sqrt{\log y}}
&\leq \sum_{\substack{p \leq x \\ h_y(p-1)> \frac{1}{y{\sqrt{\log y}}}}} \!\!\!\!\!\!\!\!\!\!\!\! h_y(p-1)
\quad+ \sum_{\substack{p \leq x \\ h_y(p+1)> \frac{1}{y{\sqrt{\log y}}}}} \!\!\!\!\!\!\!\!\!\!\!\!h_y(p+1) \\
&\leq \sum_{p \leq x} h_y(p-1) + \sum_{p \leq x} h_y(p+1)\\
&\ll \frac{\pi(x)}{y \log y },
\end{align*}
which implies \eqref{eq:E2}.

\subsection{Inconvenient primes of Type 3}

Let $\omega(n)$ denote the number of distinct prime factors of $n$ and
$\omega_y(n)$ the number of distinct prime factors $q\leq y$ of $n$.  Define
\begin{equation*}
\E_3(x) = \big\{ p \leq x \,:\, \omega_y(p^2-1) \notin [ 1.5 \log \log y, 2.5 \log\log y] \big\}.
\end{equation*}
We claim that
\begin{equation}\label{eq:E3}
\# \E_3(x) \ll \frac{\pi(x)}{ \log \log y} = o(\pi(x))
\end{equation}
as $x \to \infty$.   
Although this is essentially a result of Erd\H{o}s \cite{ErdosSolo},
we sketch a simpler proof that is easily generalized
since we later need to handle
$p^2-\ell^2$ instead of $p^2-1$.

If $p \in \E_3(x)$, then for large $x$ we have
\begin{equation*}
\omega_y(p^2-1) +1 = \omega_y(p-1) + \omega_y(p+1)
\end{equation*}
because $\gcd(p-1,p+1) = 2$.  
Thus, either
\begin{equation*}
\min\{\omega_y(p-1),\omega_y(p+1)\}  \leq  0.75\log\log y+1 \leq 0.8 \log \log y
\end{equation*}
or
\begin{equation*}
\max\{\omega_y(p-1),\omega_y(p+1)  >  1.25\log\log y > 1.2 \log \log y
\end{equation*}
for sufficiently large $x$.  
Without loss of generality, we may suppose that
\begin{equation}\label{eq:WLOGpm}
\omega_y(p-1) \leq 0.8 \log \log y
\qquad \text{or} \qquad
\omega_y(p-1) \geq 1.2 \log \log y.
\end{equation}
Then
\begin{equation}\label{eq:Rlly}
0.04(\log \log y)^2 \leq \big( \omega_y(p-1) - \log \log y)^2
\end{equation}
and similarly if $p+1$ occurs in \eqref{eq:WLOGpm}.

We next require the following ``Tur\'an--Kubilius''-type result;
see \cite[Lem.~2]{Motohashi}, \cite[\S V.5, 1, p.~159]{Sandor}.   
To study $\phi(p \pm \ell)$ for $\ell \neq 1$ requires a slight generalization; see
Lemma \ref{Lemma:MotohashiGeneral} in Section \ref{Section:General} for a statement and sketch of the proof.

\begin{lemma}[Motohashi]\label{Lemma:Motohashi}
$\displaystyle\sum_{p \leq x} ( \omega_y(p \pm 1) - \log \log y)^2 = O( \pi(x) \log \log y)$.
\end{lemma}

Now return to \eqref{eq:Rlly},
apply Lemma \ref{Lemma:Motohashi}, and conclude that
\begin{align*}
 0.04( \log \log y )^2\#\E_3(x)
&\leq \sum_{p \in \E_3(x)} ( \omega_y(p-1) - \log \log y)^2 + ( \omega_y(p+1) - \log \log y)^2 \\
&\leq \sum_{p \leq x} ( \omega_y(p-1) - \log \log y)^2 + ( \omega_y(p+1) - \log \log y)^2 \\
&= O( \pi(x) \log \log y).
\end{align*}
This yields the desired estimate \eqref{eq:E3}.

\subsection{Inconvenient primes of Type 4}
Let 
\begin{equation}\label{eq:E4def}
\E_4(x)=\Bigg\{p\leq x \,:\, \frac{p^2- 1}{\phi(p^2- 1)}>(\log y)^{1/3}\Bigg\}.
\end{equation}
We claim that
\begin{equation}\label{eq:E4}
\#\E_4(x)\ll \frac{\pi(x)}{(\log y)^{1/6}} = o(\pi(x))
\end{equation}
as $x \to \infty$.  
Since $\phi(p-1) \phi(p+1) \leq \phi(p^2-1)$, the condition
\begin{equation*}
\frac{p^2-1}{\phi(p^2-1)} > (\log y)^{1/3}
\end{equation*}
implies that 
\begin{equation*}
\frac{p+1}{\phi(p+1)} > (\log y)^{1/6} 
\quad \text{or} \quad
\frac{p-1}{\phi(p-1)} > (\log y)^{1/6}.
\end{equation*}
A standard application of the Siegel--Walfisz theorem yields
\begin{equation}\label{eq:Prachar}
\sum_{p\leq x} \frac{p-1}{\phi(p-1)} \ll \pi(x);
\end{equation}
see \cite[\S I.28, 1b, p~30]{Sandor} or \cite{Prachar}.
The same holds with $p-1$ replaced by $p+1$;
the adjustments necessary to handle $p \pm \ell$ are discussed in Section \ref{Section:General}.
Thus,
\begin{align*}
\E_4(x) (\log y)^{1/6}
&\leq \sum_{p \in \E_4(x)} \left( \frac{p-1}{\phi(p-1)} + \frac{p+1}{\phi(p+1)} \right) \\
&\leq \sum_{p \leq x} \left( \frac{p-1}{\phi(p-1)} + \frac{p+1}{\phi(p+1)} \right) \\
&\ll \pi(x),
\end{align*}
which yields \eqref{eq:E4}.

\subsection{Convenient primes}
Throughout the remainder of the proof, 
we let $5\leq p\leq x$, in which $x$ is large, and
we suppose that 
\begin{equation*}
p \notin
\E_1(x)\cup \E_2(x)\cup \E_3(x)\cup \E_4(x). 
\end{equation*}
We say that such a prime is \emph{convenient}. 
Because $\gcd(p-1,p+1)=2$, we have
\begin{equation}\label{eq:ConvenientFactor}
D_y(p^2-1)=m_1m_2,
\end{equation}
in which
\begin{equation}\label{eq:ppmmnn}
p-1=m_1n_1,\qquad
p+1=m_2n_2,\qquad
\gcd(m_1,m_2)=2, 
\end{equation}
every prime factor of $m_1 m_2$ is at most $y$, 
and every prime factor of $n_1n_2$ is greater than $y$. 
In particular, $\gcd(m_1,n_1) = \gcd(m_2,n_2)= 1$.

We claim that
\begin{equation}\label{eq:pmmpmm}
\frac{\phi(m_1)}{m_1}\neq \frac{\phi(m_2)}{m_2}
\end{equation}
for sufficiently large $x$.  In light of \eqref{eq:qqpn},
it follows that $P(n)$ is the largest prime factor of the denominator of $\phi(n)/n$.
If $\phi(m_1)/m_1=\phi(m_2)/m_2$, then $P(m_1)=P(m_2)=2$
since $\gcd(m_1,m_2)=2$.  Thus,
$m_1$ and $m_2$ are powers of $2$ and 
\begin{align*}
1
&= \omega(m_1) + \omega(m_2) - 1\\
&= \omega(m_1 m_2) 
= \omega(D_y(p^2-1))  \\
&= \omega_y(p^2-1) \in\big[1.5\log\log y,2.5\log\log y\big]
\end{align*}
because $p \notin \E_3(x)$.  
This is a contradiction for $x \geq 10^{483}$.

For convenient $p\leq x$, we have
\begin{equation*}
S(p)=\frac{\phi(p-1)}{p-1}-\frac{\phi(p+1)}{p+1}
=\frac{\phi(m_1) \phi(n_1)}  {m_1n_1}-\frac{\phi(m_2)\phi(n_2)}{m_2 n_2}.
\end{equation*}
We note that $S(p) \neq 0$ because otherwise $P(p-1) =P(p+1)$ by \eqref{eq:qqpn}.
Since $\gcd(p-1,p+1) = 2$, it would follow that $p-1$ and $p+1$ are powers of $2$,
which occurs only for $p =3$.
Lemma \ref{Lemma:Sp} ensures that $S(p)$ has the same sign as 
$\phi(p-1) - \phi(p+1)$ on a set of full density in the primes.

Since $p \notin \E_2(x)$, for large $x$ we may use the inequality
\begin{equation*}
|t + \log (1-t) | \leq |t|^2, \qquad \text{for $|t| \leq \tfrac{1}{2}$},
\end{equation*}
to obtain
\begin{align*}
\frac{\phi(n_1)}{n_1} 
& =  \prod_{\substack{r|  (p- 1)\\ r>y}} \left(1-\frac{1}{r}\right) 
= \exp\Bigg(\sum_{\substack{r|  (p- 1)\\ r>y}} \log\left(1-\frac{1}{r}\right)\Bigg)\\ 
&= \exp\Bigg(-\sum_{\substack{r|  (p- 1)\\ r>y}} \frac{1}{r}+ O\bigg( \bigg(\sum_{\substack{r|  (p- 1)\\ r>y}} \frac{1}{r} \bigg)^2\bigg)  \Bigg) \\
&= \exp\Big( -h_y(p-1) +O\big(h_y(p-1)^2\big)\Big) \\
&= 1+O\left(\frac{1}{y{\sqrt{\log y}}}\right),
\end{align*}
in which the implied constant in the preceding can be taken to be $2$.
A similar inequality holds if $n_1$ is replaced by $n_2$.  Consequently,
\begin{align}
S(p) 
& =  \frac{\phi(m_1)}{m_1}\Bigg(1+O\bigg(\frac{1}{y{\sqrt{\log y}}}\bigg)\Bigg)-\frac{\phi(m_2)}{m_2}\Bigg(1+O\bigg(\frac{1}{y{\sqrt{\log y}}}\bigg)\Bigg)\nonumber\\
& =  \frac{\phi(m_1)}{m_1}-\frac{\phi(m_2)}{m_2}+O\left(\frac{1}{y{\sqrt{\log y}}}\right),\label{eq:S}
\end{align}
in which the implied constant can be taken to be $4$.

\subsection{Weird primes}
A convenient prime $p \leq x$ is \emph{weird} if 
\begin{equation*}
S(p)\left(\frac{\phi(m_1)}{m_1}-\frac{\phi(m_2)}{m_2}\right)<0;
\end{equation*}
that is, if $S(p)$ and $\phi(m_1)/m_1-\phi(m_2)/m_2$ have opposite signs
(the second factor is nonzero if $x$ is large; see \eqref{eq:pmmpmm}).  
If this occurs, then \eqref{eq:S} tells us that
\begin{equation}\label{eq:M1M2}
\left|\frac{\phi(m_1)}{m_1}-\frac{\phi(m_2)}{m_2}\right|\,<\,\frac{4}{y{\sqrt{\log y}}}.
\end{equation}
In general, one expects the sign of $S(p)$ to
be determined by small primes; that is, those primes at most $y$.  
If $p$ is weird, then the primes $q  > y$ that divide $p^2-1$
conspire to overthrow the contribution of the small primes.

We say that a pair $(m_1,m_2)$ of positive integers is \emph{weird} if 
\begin{equation*}
24 | m_1 m _2, \quad
m_1 m_2 | L_y, \quad \gcd(m_1,m_2) = 2, \quad \text{and} \quad
\text{\eqref{eq:M1M2} holds}.
\end{equation*}
What is the reason for the appearance of the number $24$ in the preceding?
If $p\geq 5$, then considering $p^2 - 1$ modulo $3$ and $8$ reveals that $24|(p^2-1)$.
If $x \geq 10^9$, then $y  \geq 3$ and hence $P(24) = 3$ and $24|D_y(p^2-1)$.  
Consequently, if we are searching for primes $p$ for which 
$D_y(p^2-1) = m_1 m_2$, it makes sense for us to insist that $m_1m_2$ is divisible by $24$.

\begin{lemma}\label{Lemma:DLy}
Let $y\geq \exp(48^6)$ and
$D|  L_y$ be a multiple of $24$.
\begin{enumerate}
\item The number of pairs $(m_1,m_2)$ with $D=m_1m_2$ and $\gcd(m_1,m_2)=2$ is $2^{\omega(D)}$.
\item If $\omega(D)\in [1.5\log\log y,2.5\log\log y]$ and
$D / \phi(D) \leq (\log y)^{1/3}$, 
then the number of weird pairs $(m_1,m_2)$ with $D=m_1m_2$ is
\begin{equation*}
\ll \frac{2^{\omega(D)}}{\sqrt{\log\log y}}.
\end{equation*}
\end{enumerate}
\end{lemma}

\begin{proof} 
(a) In what follows, $\nu_p$ denotes the $p$-adic valuation function.
Write
\begin{equation*}
D=\prod_{q\in S} q^{\nu_q(D)},
\end{equation*}
in which $S$ is a set of primes that contains $2$, $\#S = \omega(D)$, and 
$\nu(2)\geq 3$.
Since $\gcd(m_1,m_2)=2$, it follows that
\begin{equation*}
\nu_q(m_1) =
\begin{cases}
\text{$1$ or $\nu_2(D)-1$} & \text{if $q=2$},\\
\text{$0$ or $\nu_q(D)$} & \text{if $q \geq 3$}.
\end{cases}
\end{equation*}
For each of the $\omega(D)$ primes in $S$, there are two possible choices for $\nu_q(m_1)$.
Consequently, there are $2^{\omega(D)}$ possible pairs $(m_1,m_2)$. 
\medskip

\noindent(b) Let
\begin{equation*}
(m_1,m_2)=(2n_1,2^{\nu_2(D)-1} n_2)
\quad \text{or} \quad
(2^{\nu_2(D)-1}n_1,2n_2)
\end{equation*}
and 
\begin{equation*}
(m_1',m_2')=(2n_1',2^{\nu_2(D)-1}n_2')
\quad \text{or} \quad 
(2^{\nu_2(D)-1}n_1', 2n_2')
\end{equation*}
be weird, where $n_1,n_2,n_1',n_2'$ are odd,
and let $D = m_1m_2 = m_1'm_2'$.
Suppose toward a contradiction that $n_1|  n_1'$ and $n_1<n_1'$.  
Then \eqref{eq:M1M2} says that
\begin{align}
\bigg(\frac{\phi(n_1)}{n_1}\bigg)^2
&= \frac{4 \phi(m_1)^2}{m_1^2} 
= \frac{4 \phi(m_1)}{m_1} \bigg( \frac{\phi(m_2)}{m_2} + O \bigg( \frac{1}{y \sqrt{\log y}}\bigg) \bigg)  \nonumber\\
&= \frac{4 \phi(m_1)\phi(m_2)}{m_1 m_2} + O \bigg( \frac{1}{y \sqrt{\log y}}\bigg)  \label{eq:absorb}\\
&= \frac{2 \phi(m_1m_2)}{m_1 m_2} + O \bigg( \frac{1}{y \sqrt{\log y}}\bigg) \nonumber \\
&= \frac{2 \phi(D)}{D} + O \bigg( \frac{1}{y \sqrt{\log y}}\bigg)  \nonumber\\
&= \frac{2 \phi(D)}{D}  \Bigg( 1 + O \bigg( \frac{1}{y (\log y)^{1/6}}\bigg) \Bigg) \label{eq:1over2}
\end{align}
since $D / \phi(D) \leq (\log y)^{1/3}$.
The implied constant in \eqref{eq:1over2} is $16$, in light of \eqref{eq:M1M2}
and the absorption of $4 \phi(m_1)/m_1$ in \eqref{eq:absorb}.
Similar reasoning yields an analogous expression for $\phi(n_1')/n_1'$, with the same implied constant.

Let $r$ be the smallest prime divisor of $n_1'/n_1$.  
Use the inequality
\begin{equation*}
    \qquad\frac{1+s}{1+t} \leq 1 + \frac{3}{2}|s-t| \leq 1 + \frac{3}{2}\big(|s|+|t|\big), \qquad |s|,|t| \leq \frac{1}{3},
\end{equation*}
and the fact that \eqref{eq:1over2} holds for $n_1$ and $n_1'$ to deduce that
\begin{align*}
1+ \frac{1}{r}
&< 1 + \frac{2}{r-1} 
\leq \bigg(1 + \frac{1}{r-1}\bigg)^2 \\
&= \frac{1}{(1 - 1/r)^2} 
= \bigg(\frac{n_1'}{\phi(n_1')} \cdot  \frac{ \phi(n_1)}{n_1}\bigg)^2 \\
&\leq 1 + \frac{\frac{3}{2}(16+16)}{y (\log y)^{1/6}} = 1 + \frac{48}{y (\log y)^{1/16}}.
\end{align*}
Since $r|D$ and $D|L_y$, we have $r \leq y$ and hence
\begin{equation*}
\frac{y (\log y)^{1/6}}{48} < r \leq y.
\end{equation*}
This is a contradiction if $y\geq \exp(48^6)$.

Hence, the set of odd components $n_1$ of the parts $m_1$ as $(m_1,m_2)$ ranges over weird pairs has the property that no two divide each other. 
Identifying $n_1$ with the set of its odd prime factors, 
no two $n_1$ and $n_1'$, as subsets, are contained one in another. 
Sperner's theorem\footnote{A collection of sets that does not contain $X$ and $Y$ for which $X \subsetneq Y$ is a \emph{Sperner family}. 
If $S$ is a Sperner family whose union has a total of $n$ elements, then $\# S \leq \binom{n}{ \lfloor \frac{n}{2} \rfloor}$ \cite{Sperner}.}
from combinatorics and Stirling's formula ensure that
the number of such $n_1$, and hence the number of such pairs $(m_1,m_2)$, is 
\begin{equation*}
\leq \binom{ \omega(D) }{ \lfloor \frac{\omega(D)}{2} \rfloor}
\ll \frac{2^{\omega(D)}}{\sqrt{\omega(D)}}\ll \frac{2^{\omega(D)}}{\sqrt{\log\log y}}. \qed
\end{equation*}
\end{proof}

\subsection{Conclusion}
We have shown that the number of inconvenient primes at most $x$ is $o(\pi(x))$
and hence they can be safely ignored.  Each convenient prime $p$ gives rise to a pair
$(m_1,m_2)$ as in \eqref{eq:ConvenientFactor}.  

Suppose that $x$ is large.
Let $D$ be a multiple of $24$ with 
\begin{equation}\label{eq:AcceptD}
D|L_y,\quad
\frac{D }{ \phi(D)} \leq (\log y)^{1/3} ,\quad\text{and} \quad 
\omega(D)\in [1.5\log\log y,2.5\log\log y].
\end{equation}
We wish to count the primes $p \leq x$ for which $D_y(p^2-1)=D$.
Denote this number by $\pi_D(x)$. 
To complete the proof of Theorem \ref{Theorem:Main} in the case $\ell = 1$, it suffices to show that
$(1/2+o(1))\pi_D(x)$ of primes at most $x$ have $S(p)>0$ and $(1/2+o(1))\pi_D(x)$ have $S(p)<0$, and that the implied
constant is uniform for all $D$ as above.  

Choose a pair $(m_1,m_2)$ such that 
$D=m_1m_2$ and $\gcd(m_1,m_2)=2$.
We want to count the primes $p\leq x$ such that $m_1|( p-1)$ and $m_2|(p+1)$; that is, such that
\begin{equation}\label{eq:pp11mm}
p\equiv 1\pmod {m_1}\qquad\text{and} \qquad p\equiv -1\pmod {m_2}.
\end{equation}
Apply the Chinese Remainder Theorem with moduli $\frac{1}{2}m_1,m_2$ or $m_1,\frac{1}{2}m_2$, depending upon
which of $m_1,m_2$ is exactly divisible by $2$, to see that $p$ belongs to an arithmetic progression
$a_{m_1,m_2}\pmod {D/2}$, with $\gcd(a_{m_1,m_2},D/2)=1$. 
However, we also want 
\begin{equation*}
\gcd\bigg(\frac{p^2-1}{m_1m_2},L_y\bigg)=1.
\end{equation*}
For this, we need to work modulo
\begin{equation*}
M_D:=(D/2)\prod_{r\leq y} r.
\end{equation*}
To ensure that \eqref{eq:pp11mm} holds, we do the following:

\begin{itemize}
\item If $r| D$, then we then want 
\begin{equation*}
p\equiv \varepsilon+r^{\nu_r(D/2)}\lambda \pmod {r^{\nu_r(D/2)+1}}
\end{equation*}
for some $\lambda\in \{1,2,\ldots,r-1\}$. Here, $\varepsilon=\pm 1$ according to whether $r^{\nu_r(D/2)}$ divides $m_1$ or $m_2$, respectively.
For each $r| D$, there are $r-1$ possibilities for $\lambda$ and hence there are $r-1$ possibilities for $p$ modulo $r^{\nu_r(D/2)+1}$. 

\item If $r\nmid D$, then we want $p\equiv \lambda\pmod r$ for some $\lambda\not\in \{0,1,r-1\}$.
For each $r\nmid D$, there are $r-3$ possibilities for $p$ modulo $r$. 
\end{itemize}
Thus, the number of progressions modulo $M_D$ that can contain a prime $p$ for which \eqref{eq:pp11mm} occurs is
\begin{equation}\label{eq:B}
\prod_{r| \frac{D}{2}} (r-1) \prod_{\substack{r\leq y\\ r\nmid D}} (r-3).
\end{equation}
By \eqref{eq:Lxy2}, the common modulus of all these progressions satisfies
\begin{equation*}
M_D\leq L_y^2\leq (\log x)^4
\end{equation*}
for large $x$. The Siegel--Walfisz theorem says that the number of primes in each progression is asymptotically
\begin{equation}
\label{eq:A}
\frac{\pi(x)}{\phi(M_D)} +O\left(xe^{-C{\sqrt{\log x}}}\right)
\end{equation}
for some $C>0$.
Summing up over the number of progressions 
(or, more precisely, multiplying the \eqref{eq:A} by the number of acceptable progressions \eqref{eq:B}), 
and using the fact that 
$\phi(M_D)=(D/2)\prod_{r\leq y} (r-1)$, we get a count of 
\begin{equation*}
\frac{2\pi(x)}{D} \prod_{\substack{r\leq y\\ r\nmid D}} \left(\frac{r-3}{r-1}\right)
+O\left(xL_y^2 e^{-C{\sqrt{\log x}}}\right).
\end{equation*}
The count depends on $D$ but not on the pair of divisors $(m_1,m_2)$ of $D$. We now apply Lemma \ref{Lemma:DLy} and obtain
\begin{equation}\label{eq:D}
\pi_D(x)=\frac{2^{\omega(D)+1}\pi(x)}{D} \prod_{\substack{r\leq y\\ r\nmid D}} \left(\frac{r-3}{r-1}\right)
+O(2^{\omega(D)}  xL_y^2e^{-C{\sqrt{\log x}}}).
\end{equation}
Although it is not crucial to our proof, we show in Subsection \ref{Section:Sanity} that
\begin{equation}\label{eq:piD}
\sum_{D} \pi_D(x)= \pi(x) \bigg( 1 + O\bigg( \frac{1}{\sqrt{\log \log y}}\bigg)\bigg),
\end{equation}
where the index $D$ runs over all $D$ for which \eqref{eq:AcceptD} holds, because the computation is of independent interest.

The product in \eqref{eq:D} is less than $1$ and bounded below by
\begin{align}
\prod_{\substack{r\leq y\\ r\nmid D}} \left(\frac{r-3}{r-1}\right)
&\gg \prod_{r\leq y} \left(\frac{r-3}{r-1}\right) 
= \prod_{r\leq y} \left(\frac{r^2(r-3)}{(r-1)^3}\right)\left(1 - \frac{1}{r}\right)^2  \nonumber \\
&=\prod_{r\leq y} \left(1 - \frac{3r-1}{(r-1)^3}\right) \prod_{r\leq y}\left(1 - \frac{1}{r}\right)^2  \nonumber \\
&\gg\prod_{r \leq y}\left(1-\frac{1}{r}\right)^2 \gg (\log y)^{-2} \label{eq:prodlogy2}
\end{align}
by Mertens' asymptotic formula \cite[\S VII.29.1b, p.~259]{Sandor}.
Since
\begin{equation*}
2^{\omega(D)}\geq  2^{1.5\log\log y} \gg (\log y),
\end{equation*}
we examine the main term in \eqref{eq:D} and conclude that
\begin{equation*}
\pi_D(x) \gg \frac{\pi(x)}{D (\log y)^2}.
\end{equation*}
On the other hand, the error term in \eqref{eq:D} is
\begin{align*}
O(2^{\omega(D)} xL_y^2 e^{-C{\sqrt{\log x}}})
&=O( 2^{2.5 \log \log y} x(\log x)^4 e^{-C{\sqrt{\log x}}})\\
&= O\big( (\log x)^4(\log y)^2 x e^{-C{\sqrt{\log x}}} \big)\\
&=o\left(\frac{\pi(x)}{D(\log y)^3} \right)=o\left(\frac{\pi_D(x)}{\log y}\right).
\end{align*}

There is a symmetry between non-weird pairs $(m_1,m_2)$ with 
\begin{equation*}
\frac{\phi(m_1)}{m_1}-\frac{\phi(m_2)}{m_2}>0
\qquad \text{and those with}\qquad
\frac{\phi(m_1)}{m_1}-\frac{\phi(m_2)}{m_2}<0
\end{equation*}
given by the transposition $(m_1,m_2)\mapsto (m_2,m_1)$.  Indeed,
we could return to \eqref{eq:ppmmnn} and insist that $m_1|(p+1)$
and $m_2|(p-1)$ instead.  The subsequent asymptotic estimates
go through in exactly the same manner.
Via this transposition, we obtain an asymptotically equal count between the
convenient primes $p\leq x$ 
corresponding to $(m_1,m_2)$ and the convenient primes $p\leq x$
corresponding to $(m_2,m_1)$. 
If only non-weird pairs $(m_1,m_2)$ are taken into account,
for fixed $D_y(p^2-1)=D$ this symmetry 
gives an asymptotically equal count of convenient primes $p \leq x$ with $S(p)>0$ and with $S(p)<0$.

Lemma \ref{Lemma:DLy} ensures that the
number of weird pairs $(m_1,m_2)$ with $D = m_1 m _2$ 
is $\ll 2^{\omega(D)}/{\sqrt{\log\log y}}$. 
As $x \to \infty$, we see from \eqref{eq:D} that the number of primes at most $x$
that arise from some weird pair is 
\begin{equation*}
\,\,\ll\,\, \frac{2^{\omega(D)}\pi(x)}{D\sqrt{\log\log y} } 
\prod_{\substack{r\leq y\\ r\nmid D}} \left(\frac{r-3}{r-1}\right)
=O\left(\frac{\pi_D(x)}{\sqrt{\log\log y}}\right)=o(\pi_D(x)).
\end{equation*}

Recall that the non-weird, convenient primes
have full density in the set of primes.  Of such primes $p \leq x$, the argument above shows that
an asymptotically equal amount have $S(p) > 0$ versus $S(p) < 0$ (recall that $S(p) = 0$ only for $p=3$).
This completes the proof of Theorem \ref{Theorem:Main} in the case $\ell = 1$. 
$\qed$

\subsection{Sanity check}\label{Section:Sanity}
Before extending the preceding proof to the case $\ell \geq 2$, it is helpful to
perform a quick sanity check.  
Our goal here is to prove \eqref{eq:piD}.  In light of \eqref{eq:D}, it suffices to prove that
\begin{equation}\label{eq:Sanity}
\sum_{D} \frac{2^{\omega(D)+1}}{D}\prod_{\substack{r\leq y\\ r\nmid D}} \left(\frac{r-3}{r-1}\right)=1+o(1),
\end{equation}
in which the index $D$ runs over all $D$ for which \eqref{eq:AcceptD} holds.  
In particular, \eqref{eq:piD} holds and the preceding product does not run over $r=2,3$.
These developments seems
remarkably fortuitous.  Let us provide an independent
derivation of \eqref{eq:Sanity}, which will help corroborate some of the fine details in the preceding proof.

First write $D=2^kD_1$, in which $k\geq 3$, and sum to obtain
\begin{equation*}
\bigg(\sum_{k\geq 3} \frac{2^2}{2^k} \bigg) \bigg(\sum_{{\text{$D_1$ odd}}} \frac{2^{\omega(D_1)}}{D_1}\bigg) 
= \sum_{{\text{$D_1$ odd}}}  \frac{2^{\omega(D_1)}}{D_1}.
\end{equation*}
Now write $D_1=3^k D_2$ and sum over $k\geq 1$ getting
\begin{equation*}
\bigg(\sum_{k\geq 1} \frac{2}{3^k} \bigg)\bigg(\sum_{\gcd(D_2,6)=1} \frac{2^{\omega(D_2)}}{D_2}\bigg)
=\sum_{\gcd(D_2,6)=1} \frac{2^{\omega(D_2)}}{D_2}.
\end{equation*}
For the rest, we use multiplicativity to say that the sum in \eqref{eq:Sanity} is
\begin{equation*}
\prod_{r\leq y} \left(\frac{r-3}{r-1}\right)\prod_{5\leq r\leq y} \bigg(1+\frac{2}{(r-3)/(r-1)} \sum_{k\geq 1} \frac{1}{r^k}\bigg).
\end{equation*}
However, this is not strictly correct since the sum above stops 
at the largest power $r^b\leq y$.   Moreover, the sum runs over all $D$
without restrictions such as
$\omega(D)\in [1.5\log\log y,2.5\log\log y]$ or $D/\phi(D)\leq (\log y)^{1/3}$. 
We deal with these omissions shortly.
For the time being, let us ignore these restrictions. Then the amount inside the Euler factor is 
\begin{equation*}
1+\frac{2(r-1)}{r-3} \frac{1}{r-1}=1+\frac{2}{r-3}=\frac{r-1}{r-3},
\end{equation*}
which cancels with the outside $(r-3)/(r-1)$.

Now we must examine the errors. There are essentially four types:
\begin{enumerate}
\item In each Euler factor we only sum up to $r^b$, in which $b$ is maximal such that $r^b\leq y$. 
By extending the sum to infinity we incurred an error of
\begin{equation*}
\frac{2}{(r-3)/(r-1)} \sum_{k\geq b+1} \frac{1}{r^k}=\frac{2}{r^b(r-3)}=O\left(\frac{1}{y}\right).
\end{equation*}
For $5 \leq r \leq y$, the actual Euler factor is 
\begin{equation*}
\frac{r-1}{r-3}+O\left(\frac{1}{(r-3) y}\right)=\frac{r-1}{r-3} \left(1+O\left(\frac{1}{y}\right)\right)
\end{equation*}
Similar considerations apply for $r=2,3$. The total multiplicative error is
\begin{equation*}
\left(1+O\left(\frac{1}{y}\right)\right)^{\pi(y)}=  1+O\left(\frac{\pi(y)}{y}\right)=1+O\left(\frac{1}{\log y}\right).
\end{equation*}

\item We consider only $D$ such that $D/\phi(D)\leq (\log y)^{1/3}$.
Let the set of remaining $D$ be denoted $\mathcal{D}$.  For $D \in \mathcal{D}$, we have
\begin{equation*}
\frac{1}{D}\leq \frac{1}{(\log y)^{1/3}} \cdot \frac{1}{\phi(D)}.
\end{equation*}
Applying this inequality and extending then the sum over all possible $D$, the piece of the sum over $\mathcal{D}$ is at most
\begin{equation*}
\ll \frac{\pi(x)}{(\log y)^{1/3}} \sum_{D} \frac{2^{\omega(D)+1}}{\phi(D)} \prod_{\substack{r\leq y\\ r\nmid D}} \left(\frac{r-3}{r-1}\right).
\end{equation*}
We separate out the power of $2$ in $D$ as $D=2^kD_1$ with $k\geq 3$ getting an Euler factor corresponding to $2$ of
\begin{equation*}
4\left(\frac{1}{\phi(8)}+\frac{1}{\phi(16)}+\cdots+\frac{1}{\phi(2^k)}+\cdots\right)=2.
\end{equation*}
Then we separate out a factor of $3$ from $D_1$ writing it as $D_1=3^kD_2$, getting an Euler factor corresponding to $3$ of 
\begin{equation*}
2\left(\frac{1}{\phi(3)}+\frac{1}{\phi(9)}+\cdots+\frac{1}{\phi(3^k)}+\cdots\right)=\frac{3}{2}.
\end{equation*}
For the remaining primes $r\geq 5$, we form the Euler product getting
\begin{equation*}
\frac{\pi(x)}{(\log y)^{2/3}} \prod_{5\leq r\leq y} \left(\frac{r-3}{r-1}\right) \prod_{r=5}^y \left(1+\frac{2}{(r-3)/(r-1)} \sum_{b\geq 1} \frac{1}{\phi(r^b)}\right).
\end{equation*}
The factor inside the parentheses is  
\begin{align*}
1+\frac{2}{(r-3)/(r-1)} \frac{1}{(r-1)(1-1/r)} 
& =  1+\frac{2r}{(r-1)(r-3)}\\
& =  \frac{r^2-2r+3}{(r-1)(r-3)}\\
& =  \frac{(r-1)}{(r-3)} \left(1+\frac{2}{(r-1)^2}\right).
\end{align*}
Multiply this by the outside factor $(r-3)/(r-1)$ and get
\begin{equation*}
\bigg(\frac{r-3}{r-1}\bigg) \left(1+\frac{2}{(r-3)/(r-1)} \sum_{b\geq 1} \frac{1}{\phi(r^b)}\right)=1+\frac{2}{(r-1)^2}.
\end{equation*}
Taking the product of the factors above over $r\in [5,y]$, we get a convergent product. Consequently,
\begin{equation*}
\pi(x) \!\!\!\!\!\!\!\!\!
\sum_{\substack{D\mid L_y, 24\mid D\\ D/\phi(D)>(\log y)^{1/3}}}  \!\!\!\!\!\!\!\!
 \frac{2^{\omega(D)+1}}{D}\prod_{\substack{5\leq \leq y \\ r\nmid D}} \left(\frac{r-3}{r-1}\right)
 \quad \ll \quad  \frac{\pi(x)}{(\log y)^{1/3}}.
\end{equation*}

\item We need to consider $D$ with $\omega(D)\not\in [1.5\log\log y,2.5\log\log y]$. 
From the preceding material and \eqref{eq:prodlogy2}, we have
\begin{equation*}
\prod_{5\leq r\leq y} \left(1+\sum_{b\geq 1} \frac{2}{(r-3)/(r-1)} \sum_{b\geq 1} \frac{1}{r^b}\right)
\asymp \left(\prod_{5\leq r\leq y} \left(\frac{r-3}{r-1}\right)\right)^{-1}\asymp (\log y)^2.
\end{equation*}
We first deal with $D$ with many prime factors. Consider the multiplicative function defined for prime powers $r^b$ with $r\geq 5$ by 
\begin{equation*}
f(r)=\frac{2(r-1)}{(r-3)r^b}.
\end{equation*}
If $K:=\lfloor 2.5\log\log y\rfloor$, then
\begin{equation*}
\sum_{\substack{r\mid D_2\!\implies\! r \in [5,y]\\ \omega(D_2)>K-2}} \!\!\!\!\!\!\!\! f(D_2)
\quad \leq \quad \sum_{k> K-2} \frac{1}{k!} \bigg(\sum_{\substack{r\in [5,y]\\ b\geq 1}} f(r^b)\bigg)^k.
\end{equation*}
We have 
\begin{align*}
S 
& =  \sum_{\substack{r\in [5,y]\\ b\geq 1}} f(r^b)
 =  \sum_{\substack{b\geq 1\\ 5\leq r\leq y}} \left(\frac{2}{r^b}+\frac {2}{r^b}\left(\frac{r-1}{r-3}-1\right)\right)\\
& =  \sum_{\substack{r\in [2,y]\\ b\geq 1}} \frac{2}{r^b}+O\bigg(\sum_{\substack{r\geq 5\\ b\geq 1}} \frac{1}{r^{b+1}}\bigg)\\
& =  
2\log\log y+O(1),
\end{align*}
in which we have used Mertens' theorem \cite[\S VII.28.1b]{Sandor}.
In the sum $\sum_{k>K} S^k/k!$, the ratio of two consecutive terms is
\begin{equation*}
\frac{S^{k+1}/(k+1)!}{S^k/k!}=\frac{S}{k+1}=\frac{2\log\log y+O(1)}{2.5\log\log y+O(1)}<\frac{5}{6}
\end{equation*}
for $k>K-2$ and large $x$, so the first term dominates. 
With $K!>(K/e)^K$, the contribution of $D$ with $\omega(D)>2.5\log\log y$ is at most
\begin{equation*}
\ll \frac{S^{K-2}}{(K-2)!} \leq \left(\frac{2e\log\log y+O(1)}{2.5\log\log y+O(1)}\right)^{2.5\log\log y+O(1)}\ll (\log y)^c,
\end{equation*}
in which $c=2.5\log(2e/2.5)<1.95$. Multiplying this by (see \eqref{eq:prodlogy2})
\begin{equation*}
\prod_{5\leq r\leq y} \left(\frac{r-3}{r-1}\right)\ll (\log y)^{-2},
\end{equation*}
we obtain
\begin{equation*}
\pi(x)\sum_{\substack{D\mid L_y \\ 24\mid D\\ \omega(D)>2.5\log\log y}} 
\frac{2^{\omega(D)+1}}{D} \prod_{\substack{5\leq r\leq y\\ r\nmid D}} \left(\frac{r-3}{r-1}\right)\ll \frac{\pi(x)}{(\log y)^{0.5}}.
\end{equation*}

We use a similar argument for $D$ with $\omega(D)<1.5\log\log y$. 
In this case, let $K_1:=\lfloor 1.5\log\log y\rfloor$.  We have to deal with
\begin{equation*}
\sum_{\substack{D_2: r\mid D_2\!\implies\! r\in [5,y]\\ \omega(D_2)<1.5\log\log y-2}}\!\!\!\!\!\!\!\! f(D_2)
\quad\leq\quad \sum_{k\leq K_1-2} \frac{1}{k!} S^k.
\end{equation*}
For $k>K_1-2$ and large $x$, the ratio of any two consecutive terms above is 
\begin{equation*}
\frac{S^{k+1}/(k+1)!}{S^k/k!}=\frac{S}{k+1}>\frac{2\log\log y+O(1)}{1.5\log\log y+O(1)}>\frac{5}{4},
\end{equation*}
it follows that the  last term dominates. Thus, this sum is at most
\begin{equation*}
\left(\frac{2e\log\log y+O(1)}{K_1-2}\right)^{K_1-2}
=\left(\frac{2e\log\log y+O(1)}{1.5\log\log y+O(1)}\right)^{1.5\log\log y+O(1)}\ll (\log y)^{c_1},
\end{equation*}
in which where $c_1=1.5\log(2e/1.5) <1.95$.
Consequently, the contribution of $D$ with $\omega(D)<1.5\log\log y$ to the sum defining $\pi_D(x)$ is 
\begin{equation*}
\ll \frac{\pi_D(x)}{(\log y)^{0.05}}.
\end{equation*}
\end{enumerate}  

Putting everything together we obtain \eqref{eq:Sanity}, which is equivalent to \eqref{eq:piD}.

\section{Proof of Theorem \ref{Theorem:Main} for $\ell \geq 2$}\label{Section:General}

The proof for $\ell \geq 2$ follows largely on the lines of the case $\ell = 1$,
although there are a number of minor adjustments that must be made.  For example, in 
Lemma \ref{Lemma:DLy} we assumed that $D$ is a multiple of $24$.  Elementary
considerations reveal that the following adjustments are necessary for various values of $\ell$:
\begin{enumerate}
\item $D$ is coprime to all primes that divide $\ell$.
\item $D$ is odd if $\ell$ is even.
\item $D$ is a multiple of $8$ if $\ell$ is odd.
\item $D$ is a multiple of $3$ if and only if $\ell$ is not.
\end{enumerate}
More significant modifications are discussed below.

\subsection{The case of equality}
We need a variant of the inequality \eqref{eq:PP11Pix}.
The estimate provided by the following theorem involves two special cases.
Numerical evidence strongly suggests that this distinction is not simply a 
byproduct of our proof; see Table \ref{Table:Equality}.  If we replace the
use of the Brun sieve in what follows with an appeal to the Bateman--Horn
conjecture \cite{Bateman}, then the larger of the two upper bounds becomes
an asymptotic equivalence if the appropriate constant factor is introduced.

\begin{theorem}\label{Theorem:Equality}
For each $\ell\geq 1$,
\begin{equation*}
\# \big\{ p \leq x : \phi(p-\ell) = \phi(p+\ell) \big\} 
\,\, \ll \,\,
\begin{cases}
\dfrac{x}{(\log x)^3} & \text{if $\ell = 4^n-1$},\\[10pt]
\dfrac{x}{e^{(\log x)^{1/3}}} & \text{otherwise}.
\end{cases}
\end{equation*}
\end{theorem}

\begin{proof}
In Lemma \ref{Lemma:GHP}, let $k = 2\ell$ and $n = p-\ell$, in which $p$ is prime.
Suppose that $j$ and $j+2\ell$ have the same prime factors.  Since
\begin{equation}\label{eq:pjl}
p =  j\left( \frac{(j+2\ell)t}{g} + 1 \right) + \ell,
\end{equation}
it follows that $j$ is not divisible by any prime factor of $\ell$ and hence
$g = \gcd(j,2\ell) = 2$.
Thus, $j = 2^{m}$
and $j+2\ell = 2^{m + n}$ for some $m,n \geq 1$.  
Then $2^{m+n} = 2^m + 2\ell$ and
\begin{equation*}
\ell = 2^{m-1}(2^{n} - 1) .
\end{equation*}
If $m \geq 2$, then $\ell$ is even and \eqref{eq:pjl} implies that $2|p$, a contradiction.
Thus, the upper bound from Lemma \ref{Lemma:GHP} applies in this case.

If $m=1$, then $j = 2$ and $\ell = 2^n-1$.  Then
\begin{equation*}
p =  2(\ell+1)t + (\ell+2),\qquad
q = t+1 \qquad \text{and} \qquad r = (\ell+1)t + 1
\end{equation*}
and we count $t \leq \frac{x-(\ell+2)}{2(\ell+1)} \sim \frac{x}{2\ell}$ for which $p,q,r$ 
are simultaneously prime.  Let
\begin{equation*}
f_1(t) = 2^{n+1}t + (2^n+1), \qquad
f_2(t) = t+1, \quad \text{and} \quad
f_3(t) = 2^n t+1.
\end{equation*}
If $n$ is odd, then 
\begin{equation*}
f_1(0) \equiv f_2(2) \equiv f_3(1)  \equiv 0 \pmod{3}.
\end{equation*}
There are three possibilities:
\begin{enumerate}
\item If $f_1(0) = 3$, then $f_2(0) = 1$ is not prime and no prime triples are produced.
\item If $f_2(2) = 3$, then for each odd $n$, at most one prime triple is produced.\footnote{The only 
    odd $n <99$ for which a prime triple arises in this manner are
    $n=1,3,7,15$, from which we obtain the triples $(11,3,5)$, $(41,3,17)$, $(641,3,257)$,
    and $(163841,3,65537)$.}
\item If $f_3(1) = 3$, then $n = 1$ and only the prime triple $(7,2,3)$ is produced.
\end{enumerate}
In each case, the upper bound from Lemma \ref{Lemma:GHP} dominates.

If $n$ is even, then
none of $f_1,f_2,f_3$ vanish identically modulo any prime.
The Brun sieve says that the number of $p \leq x$
for which $p,q,r$ are prime is $O(x/ (\log x)^3)$, which dominates
the estimate from Lemma \ref{Lemma:GHP}.
\end{proof}

\subsection{A more general comparison lemma}

The next adjustment that is required is an analogue of the comparison lemma (Lemma \ref{Lemma:Sp}).
This turns out to be more involved than expected.  In fact, we first need a generalization
of the ``Tur\'an--Kubilius''-type result from Lemma \ref{Lemma:Motohashi}.  Since this is a
minor variant of an existing result, we only sketch the proof.

\begin{lemma}\label{Lemma:MotohashiGeneral}
For each $\ell\geq 1$,
$\displaystyle\sum_{p \leq x} ( \omega_y(p \pm \ell) - \log \log y)^2 = O( \pi(x) \log \log y)$.
\end{lemma}

\begin{proof}
Since $y \leq (\log x)^2$, apply the Siegel--Walfisz theorem to obtain
\begin{align*}
\sum_{p\leq x} \omega_y(p\pm \ell) &= \pi(x) \log \log y + O(\pi(x)), \\
\sum_{p\leq x} \omega_y(p\pm \ell)^2 &= \pi(x) (\log \log y)^2 + O(\pi(x) \log \log y),
\end{align*}
in which the $\log \log y$ term arises from an application of Mertens' theorem \cite[\S VII.28.1b]{Sandor}.  
Now expand $\sum_{p \leq x} ( \omega_y(p \pm \ell) - \log \log y)^2$ and apply the preceding.
\end{proof}

The direct generalization of Lemma \ref{Lemma:Sp} for $\ell \geq 2$ runs into trouble.
If $\ell$ is sufficiently large, then the $+1$ in \eqref{eq:GoingWrong} becomes too large
for the same argument to work.  The evenness of $\phi(p-\ell) - \phi(p+\ell)$ is no longer
sufficient to push the argument through.  Fortunately, we are able to employ the following lemma instead.

\begin{lemma}\label{Lemma:2L}
For $\ell,m \geq 1$, 
\begin{equation*}
\# \big\{ p \leq x \,:\, \phi(p\pm \ell) \equiv 0 \pmod{2^m} \big\} \,\,\sim\,\, \pi(x).
\end{equation*}
\end{lemma}

\begin{proof}
Fix $\ell \geq 1$.
Let $\omega(n)$ denote the number of distinct prime divisors of $n$.
Then $2^{\omega(n)-1} | \phi(n)$ since
$\phi(n) = \prod_{p^a\| n} p^{a-1}(p-1)$.
If 
\begin{equation}\label{eq:Key}
2+ \log_2 \ell \leq m +1 \leq  \omega(p\pm \ell),
\end{equation}
then
\begin{equation*}
2 \ell \leq 2^{m} \leq 2^{\omega(p \pm \ell)-1}
\qquad\text{and hence}\quad
\phi(p\pm \ell) \equiv 0 \pmod{2^m}.
\end{equation*}
Thus, it suffices to show that the set of primes $p \leq x$
for which \eqref{eq:Key} fails has a counting function that is $o(\pi(x))$.
Let $x$ be so large that $y = \log \log x$ satisfies
\begin{equation*}
2 + \log_2 \ell \leq \tfrac{1}{2} \log \log y
\end{equation*}
and let
\begin{equation*}
\E(x) = \big\{ p \leq x \,:\,  \omega(p- \ell) < 2 + \log_2 \ell \quad \text{or}\quad
\omega(p+ \ell) < 2 + \log_2 \ell\big\} 
\end{equation*}
Let $\omega_y(n)$ denote the number of distinct prime factors $q \leq y$ of $n$.
If $p \in \E(x)$, then
\begin{equation*}
\omega_y(p - \ell) <\tfrac{1}{2}\log \log y
\qquad\text{or}\qquad
\omega_y(p + \ell) <\tfrac{1}{2}\log \log y,
\end{equation*}
and hence
\begin{equation*}
\omega_y(p - \ell) \notin [ \tfrac{1}{2} \log \log y, \tfrac{3}{2} \log\log y]
\qquad\text{or}\qquad
\omega_y(p + \ell) \notin [ \tfrac{1}{2} \log \log y, \tfrac{3}{2} \log\log y].
\end{equation*}
Then
\begin{equation*}
\tfrac{1}{4}(\log \log y)^2 \leq \big( \omega_y(p- \ell) - \log \log y)^2
\quad\text{or}\quad
\tfrac{1}{4}(\log \log y)^2 \leq \big( \omega_y(p+ \ell) - \log \log y)^2.
\end{equation*}
Lemma \ref{Lemma:MotohashiGeneral} ensures that
\begin{align*}
\tfrac{1}{4}( \log \log y )^2\#\E(x)
&\leq \sum_{p \in \E(x)} ( \omega_y(p-\ell) - \log \log y)^2 + ( \omega_y(p+\ell) - \log \log y)^2 \\
&\leq \sum_{p \leq x} ( \omega_y(p-\ell) - \log \log y)^2 + ( \omega_y(p+\ell) - \log \log y)^2 \\
&= O( \pi(x) \log \log y).
\end{align*}
Thus,
\begin{equation*}
\#\E(x) \ll \frac{\pi(x)}{\log \log y} = o(\pi(x)). \qed
\end{equation*}
\end{proof}

Our replacement for the comparison lemma is the following.
Since the exceptional set is $o(\pi(x))$, it will not affect
the proof of Theorem \ref{Theorem:Main} in the case $\ell \geq 2$.

\begin{lemma}\label{Lemma:Spell}
For each $\ell \geq 1$,
the set of primes $p$ for which
$\phi(p-\ell) - \phi(p+\ell)$ and
\begin{equation}\label{eq:Sp}
S(p) := \frac{\phi(p-\ell)}{p-\ell} - \frac{\phi(p+\ell)}{p+\ell}
\end{equation}
have the same sign has counting function 
$\sim \pi(x)$.
\end{lemma}

\begin{proof}
By Theorem \ref{Theorem:Equality}, 
it suffices to show that the set of primes $p$ for which
\begin{equation}\label{eq:Trouble}
\phi(p-\ell ) \,> \,\phi(p+\ell )
\quad\iff\quad
\frac{\phi(p-\ell )}{p-\ell }\, > \,\frac{\phi(p+\ell )}{p+\ell }
\end{equation}
has counting function $\sim \pi(x)$.
The forward implication is straightforward, so we focus on the reverse.
If the inequality on the right-hand side of \eqref{eq:Trouble} holds, then
\begin{align*}
0 
&< (p+\ell) \phi(p - \ell) - (p - \ell) \phi(p + \ell) \\
&= p [\phi(p-\ell) - \phi(p+\ell)]  + \ell \phi(p-\ell) + \ell \phi(p + \ell) \\
&\leq p [\phi(p-\ell) - \phi(p+\ell)]  + \ell (p-\ell) + \ell (p + \ell) \\
&= p [\phi(p-\ell) - \phi(p+\ell) + 2 \ell].
\end{align*}
Let $2^m > 2\ell$ and apply lemma 
Lemma \ref{Lemma:2L} to conclude that 
$\phi(p-\ell) - \phi(p+\ell) > 0$
for $p$ in a set with counting function $\sim \pi(x)$.
\end{proof}

\subsection{Final ingredient}
The only other ingredient necessary to consider $\ell \geq 2$
is a replacement for the estimate \eqref{eq:Prachar}.
We include the proof for completeness.

\begin{lemma}\label{Lemma:Reciprocal}
$\displaystyle \sum_{p\leq x} \frac{p \pm \ell }{\phi(p\pm \ell )} \ll \pi(x)$.
\end{lemma}

\begin{proof}
Let $\sigma(n)$ denote the sum of the divisors of $n$.  Then 
\begin{equation*}
\frac{\sigma(n)}{n} = \sum_{d|n} \frac{1}{d}
\qquad \text{and} \qquad
\frac{6}{\pi^2} < \frac{\sigma(n) \phi(n)}{n^2},
\quad \text{for $n \geq 2$};
\end{equation*}
see \cite[\S I.3.5]{Sandor}.  The Siegel--Walfisz theorem provides $C>0$ so that
\begin{align*}
\sum_{p \leq x} \frac{p \pm \ell }{ \phi(p \pm \ell ) }
&\ll \sum_{p\leq x} \frac{\sigma(p \pm \ell )}{(p \pm \ell )} 
=\sum_{p\leq x} \sum_{d|(p\pm \ell )} \frac{1}{d} 
=\sum_{d\leq x} \frac{1}{d} \!\!\!\!\!\!\!\!\sum_{\substack{p\leq x\\ p\equiv \mp 1 \pmod{d}}} \!\!\!\!\!\!\!\!1 
= \sum_{d \leq x} \frac{ \pi(x;\mp \ell,d) }{d}\\
&=\sum_{d \leq (\log x)^3} \frac{ \pi(x;\mp 1,d) }{d} + \sum_{(\log x)^3 \leq d \leq x} \frac{ \pi(x;\mp 1,d) }{d} \\
&\leq \sum_{d \leq (\log x)^3}\bigg( \frac{ \pi(x)}{d\phi(d)} + O\Big(x e^{-C \sqrt{\log x}}\Big) \bigg)\quad + \sum_{(\log x)^3 \leq d \leq x} \frac{x }{d^2} \\
&\leq \pi(x) \sum_{1 \leq d < \infty}\frac{1}{d\phi(d)} + O\Big(x (\log x)^3 e^{-C \sqrt{\log x}}\Big) \quad +\quad x \!\!\!\!\!\!\!\!\sum_{(\log x)^3 \leq d < \infty} \frac{1 }{d^2} \\
&\ll \pi(x)  + x (\log x)^3 e^{-C \sqrt{\log x}} + \frac{x}{(\log x)^3 }\\
&\ll \pi(x). \qed
\end{align*}
\end{proof}

\noindent\textbf{Acknowledgment.} We thank the referee for suggestions which improved the quality of our paper.

\bibliographystyle{amsplain}

\bibliography{OVETNPA}

\end{document}